\documentclass[11pt]{article}
\usepackage{latexsym}  
\usepackage{enumerate}
\usepackage{amsmath,amssymb,amsfonts,graphicx,tikz,tikz-cd}
\usepackage{enumitem}

\DeclareMathOperator{\Rel}{\sf Rel}
\DeclareMathOperator{\PT}{\sf PT}
\DeclareMathOperator{\C}{\mathcal C}
\DeclareMathOperator{\Se}{\mathcal S}

\begin{document}

\title{Ordered Ehresmann semigroups and categories}
\author{T. Stokes}

\date{}
\maketitle

\newcommand{\bea}{\begin{eqnarray*}}
\newcommand{\eea}{\end{eqnarray*}}

\newcommand{\ben}{\begin{enumerate}}
\newcommand{\een}{\end{enumerate}}

\newcommand{\bi}{\begin{itemize}}
\newcommand{\ei}{\end{itemize}}

\newenvironment{proof}{\noindent \textbf{Proof.}\hspace{.7em}}
                   {\hfill $\Box$
                    \vspace{10pt}}

\newcommand{\mc}{\mathcal}

\newcommand{\dom}{\mbox{dom}}
\newcommand{\ran}{\mbox{ran}}

\newtheorem{thm}{Theorem}[section]
\newtheorem{theorem}[thm]{Theorem}
\newtheorem{lem}[thm]{Lemma}
\newtheorem{pro}[thm]{Proposition}
\newtheorem{dfnpro}[thm]{Definition/Proposition}
\newtheorem{proposition}[thm]{Proposition}
\newtheorem{cor}[thm]{Corollary}
\newtheorem{conj}[thm]{Conjecture}
\newtheorem{corollary}[thm]{Corollary}
\newtheorem{eg}[thm]{Example}
\newtheorem{dfn}[thm]{Definition}

\begin{abstract}
Ehresmann semigroups may be viewed as biunary semigroups equipped with domain and range operations satisfying some equational laws.  Motivated by some of the main examples, we here define ordered Ehresmann semigroups, and consider their basic properties as well as special cases in which the order is algebraically definable.  In particular, one and two-sided restriction semigroups equipped with their natural orders are characterised within the class of ordered Ehresmann semigroups.  The main result is an ESN-style theorem for ordered Ehresmann semigroups with particular reference to the special cases. 
\end{abstract}

\medskip
 
\medskip

\section{Introduction}

The Ehresmann-Nambooripad-Schein Theorem (``ESN Theorem") for inverse semigroups gives an isomorphism between the category of inverse semigroups and a suitable category of partial algebras.  The main significance of this result is that one can capture the entire semigroup by limited knowledge of the product: just the compatible products (defined in some way), certain products of idempotents and arbitary elements (``restriction" and ``corestriction"), and the product of idempotents with one-another (a semilattice operation); one also retains the natural order determined by the inverse semigroup.  The associated partial algebra is then a small category with additional order-theoretic structure, making it a so-called ``inductive groupoid", from which the original inverse semigroup may be reconstructed.  Moreover every inductive groupoid arises from an inverse semigroup in this way, the correspondence is one-to-one, and there is a suitable category isomorphism between the two classes.

Onbe of the most successful generalisations of this result to possibly non-regular semigroups is given by Lawson in \cite{law1}, applying to what the author calls Ehresmann semigroups.  These may be defined in various ways, the one adopted here being as biunary semigroups equipped with unary operations here denoted by $D,R$ that model domain and range in a certain sense.  Examples include (in increasing generality): inverse semigroups in which one defines $D(s)=ss'$ and $R(s)=s's$ for all $s$ (where $s'$ is the inverse of $s$), (two-sided) restriction semigroups, and semigroups embeddable in the semigroup $\Rel(X)$ of binary relations on some set $X$ under relational composition, in which  for all $\rho\in \Rel(X)$,
$$D(\rho)=\{(x,x)\mid x\in \dom(\rho)\},\ R(\rho)=\{(x,x)\mid x\in \ran(\rho)\}.$$
In an Ehresmann semigroup $S$, there is a distinguished semilattice of idempotents $D(S)$, consisting of all elements of the form $D(s)$ (equivalently, $R(s)$), and this replaces the semilattice of all idempotents used in the original ESN theorem.

In \cite{law1}, Lawson showed how to obtain a suitable type of small category from an Ehresmann semigroup $S$ that retains only those products  $st$ where $R(s)=D(t)$, retains $D$ and $R$, and retains not one but two partial orders determined by $S$.  Defining $s\leq_l t$ when $s=D(s)t$, and $s\leq_r t$ when $s=tR(s)$, both $\leq_l$ and $\leq_r$ are partial orders, each agreeing with the semilattice order on $D(S)$ when restricted to it.  One can use these two partial orders to characterise products such as $es,sf$ where $e,f\in D(S)$ are such that $e\leq D(s)$, $f\leq R(s)$ respectively: $es$ is the unique $t\in S$ such that $t\leq_l s$ and $D(t)=e$, and dually for $sf$.  One can then characterise precisely those enriched categories arising in this way: they have two partial orders satisfying certain conditions, and notions of restriction and corestriction defined via these partial orders, which allow products of the form $es,sf$ to be formed, and thence products of arbitrary category elements so that multiplication is everywere-defined.  Lawson characterised the class of small categories arising in this way and called them Ehresmann categories in \cite{law1}; he then proceeded to obtain a category isomorphism between Ehresmann semigroups (with morphisms being homomorphisms preserving $D$ and $R$) and Ehresmann categories (equipped with suitable morphisms).  

Much more recently,  in \cite{law2} Lawson took a different and rather simpler approach.  Instead of defining partial orders to augment the category structure determined by an Ehresmann semigroup $S$, he supplemented the category with a so-called biaction of the identities of the category on the arrows in general, thereby allowing semigroup products of the form $es,sf$ ($e,f\in D(S), s\in S$)  to be captured; in particular, products of the form $ef$ (where $e,f\in D(S)$) are captured, giving a semilattice structure on the projections.  Lawson characterised the categories with biaction arising in this way, calling them categories with Ehresmann biaction.  This approach requires less structure and complexity of axioms than that presented in \cite{law1}, dispensing with partial orders entirely, but comes at the expense of requiring all products of the form $es,se$ to be retained rather than only some (although it is not hard to slightly reformulate things so that one need only consider products of the form $es$ where $e\leq D(s)$ and dually).  A very similar approach was used by Fitzgerald and Kinyon in \cite{fitzkin}, where slightly greater generality was used: the distinguished semilattice is replaced by a band.

Lawson's approach to obtaining an ESN-style theorem for Ehresmann semigroups used in \cite{law2} may be far simpler than that used in \cite{law1}, but we believe the general ideas in \cite{law1} have much value, particularly if the Ehresmann semigroups one is interested in are themselves ordered.  In fact many natural examples of Ehresmann semigroups are ordered in a way compatible with the operations: inverse and restriction semigroups by their natural order, and $\Rel(X)$ by the order of set inclusion (which is not algebraically definable).  

When the Ehresmann semigroups under consideration are ordered, this has immediate ramifications for the the associated categories in any ESN-style theorem one seeks: such categories will necessarily be ordered.  One could proceed by attempting to ``bolt on" the additional order structure to either an Ehresmann category (which already has two partial orders) or a category with Ehresmann biaction (which has none).  However, we shall see that one can work directly with the order without requiring the additional structure of either an Ehresmann category or a category with Ehresmann biaction -- such additional structure may be derived but need not be assumed.  Special cases may be obtained corresponding to restriction semigroups, Ehresmann semigroups that are one-sided restriction semigroups, and their common generalisation, the so-called de Barros semigroups (defined in \cite{law1}).

In this work, we first review the details of the earlier work on Ehresmann semigroups and ESN-style theorems for them, especially as it appears in \cite{law1} and \cite{law2}.  We then define ordered Ehresmann semigroups to be those admitting a so-called Ehresmann order, giving examples and obtaining their basic properties.  It is shown that a given Ehresmann semigroup may admit more than one Ehresmann order or possibly none at all.  We consider special cases of Ehresmann semigroups in which the Ehresmann order may be algebraically definable, giving alternative characteristions of these classes as ordered Ehresmann semigroups.

The final section concerns our ``ESN Theorem".  We use a combination of some of the laws considered in \cite{law1}, together with a new law, to define the appropriate types of small categories with order, then obtain the hoped-for category isomorphism.  We then obtain special cases of the correspondence applying to the particular classes of Ehresmann semigroups with algebraically definable Ehresmann order considered in the previous section.

Throughout, all function and relation compositions will be written using the left-to-right convention, and functions will be written on the right of their arguments, aside from unary operations which will be written on the left ($D(s)$ rather than $sD$, and so on).

\section{Preliminaries}

\subsection{Localisable and Ehresmann semigroups}

We start with slightly greater generality than Ehresmann semigroups.  As in \cite{fitzkin}, a {\em localisable semigroup} is a biunary semigroup $(S,\cdot,D,R)$ satisfying the following laws:
\ben[label=\textup{(L\arabic*)}]
\item $D(s)s=s, sR(s)=s$;
\item $D(R(s))=R(s), R(D(s))=D(s)$;  \label{L2}
\item $D(st)=D(sD(t)), R(st)=R(R(s)t)$; \label{L3}
\item $D(D(s)D(t))=D(s)D(t)$.  \label{L4}
\een
These are not exactly the laws given in \cite{fitzkin}: the final law is different but follows from those given in \cite{fitzkin} since $D(S)=\{D(s)\mid s\in S\}=R(S)$ is a band in any localisable semigroup.  Conversely, the law $D(D(s)t)=D(s)D(t)$, appearing in \cite{fitzkin}, follows easily from the above laws since $D(D(s)t)=D(D(s)D(t))=D(s)D(t)$ using \ref{L2} and \ref{L4}; dually for the law $R(sR(t))=R(s)R(t)$.

A localisable semigroup is {\em Ehresmann} if it also satisfies 
\bi
\item $D(s)D(t)=D(t)D(s)$.
\ei

If $(S,\cdot,D,R)$ is localisable, we call elements of $D(S)=R(S)$ {\em projections}; they form a band under multiplication, which is a semilattice if (and only if) $(S,\cdot,D,R)$ is Ehresmann.  (Ehresmann semigroups were first defined in \cite{law1} simply as semigroups, with the unary operations secondary, but here we assume the unary operations are part of the signature throughout.)  

Examples of Ehresmann semigroups include any two-sided restriction semigroup, hence any inverse semigroup $(S,\cdot,')$ in which one defines $D(s)=ss'$ and $R(s)=s's$, where $s'$ is the inverse of $s\in S$.  A more generic example is $\Rel(X)$, the semigroup of binary relations on $X$ equipped with relational composition (written left-to-right), and unary operations $D$ and $R$, where 
$$D(\rho)=\{(x,x)\mid x\in \dom(\rho)\},\ R(\rho)=\{(x,x)\mid x\in \ran(\rho)\}.$$
Hence, the projections are precisely the restrictions of the identity map on $X$.  Some examples of localisable semigroups that are not Ehresmann are given in \cite{fitzkin}.

\begin{dfn}
Let $(S,\cdot,D,R)$ and $(T,\cdot,D,R)$ be Ehresmann semigroups.  We say $F:S\rightarrow T$ is an {\em Ehresmann semigroup homomorphism} if it is a semigroup homomorphism which respects $D$ and $R$: for all $s\in S$, $D(s)F=D(sF)$ and $R(s)F=R(sF)$.
\end{dfn}

This notion of Ehresmann semigroup homomorphism coincides with that of an admissible homomorphism considered in \cite{law1}, as shown there (in the discussion prior to Lemma 4.22).  It also coincides with the notion of morphism for localisable semigroups given in \cite{fitzkin}.

\subsection{ESN-style theorems for Ehresmann semigroups}  \label{ESNEhrs}

A (small) category here refers to a partial algebra $(C,\circ,D,R)$, in which $\circ$ is a partial binary operation and $D,R$ are unary domain and range operations such that the following laws are satisfied for all $x,y,z\in C$:
\bi
\item $D(x)\circ x=x=x\circ R(x)$;
\item $D(R(x))=R(x)$, $R(D(x))=D(x)$;
\item $x\circ y$ is defined if and only if $R(x)=D(y)$;
\item if $x\circ y$ is defined then $D(x\circ y)=D(x)$, $R(x\circ y)=R(y)$;
\item if $(x\circ y)\circ z$ and $x\circ (y\circ z)$ are defined then they are equal.
\ei
This is the approach taken in \cite{fitzkin} and is a common one in algebra, in which objects are not needed; one can introduce an object for each element of the set of identities $D(C)=R(C)=\{D(x)\mid x\in C\}$ of $C$, and then the above definition is easily seen to be equivalent to the usual one for small categories.

\begin{dfn}  \label{pp}
Given an Ehresmann semigroup $(S,\cdot,D,R)$, define the partial product $s\circ t$ to equal $st$ but to only be defined when $R(s)=D(t)$.
\end{dfn}

This is precisely what is done in \cite{law1} and \cite{law2}, and for the larger class of localisable semigroups in \cite{fitzkin}.  The resulting partial algebra $(S,\circ,D,R)$ is a category which is used in various ways in these works to obtain ESN-style theorems for Ehresmann and localisable semigroups.  

The following definition is from \cite{law1}, but follows earlier work due to Ehresmann (as do most of the subsequent properties for categories considered below -- see \cite{Ehres}).

\begin{dfn}\label{defn:orderedconst}
Let $C$ be a set. We say $(C,\circ,D,R, \leq)$ is an {\em $\Omega$-structured category} if it satisfies the following axioms on $C$: 
\ben[label=\textup{(OC\arabic*)}]
\item $(C,\circ,D,R)$ is a category and $(C,\leq)$ is a partially ordered set;
\item if $a\leq b$ then $D(a)\leq D(b)$ and $R(a)\leq R(b)$;
\item if $a\leq b$ and $c\leq d$ with both $a\circ c$ and $b\circ d$ existing, then $a\circ c\leq b\circ d$.
\een
\end{dfn}

The definition of ordered category given in \cite{law1} also requires the following law:
\ben[label=\textup{(OC4)}]
\item if $a\leq b$, $D(a)=D(b)$, and $R(a)=R(b)$ then $a=b$.
\een

Note that (OC4) can be expressed as saying that $\leq$ is trivial on hom-sets.  It can be usefully strengthened in two different ways:

\ben[label=\textup{(OC4A)}]
\item for all $s,t\in C$ if $s\leq t$ and $D(s)=D(t)$, then $s=t$;
\een
and 
\ben[label=\textup{(OC4B)}]
\item for all $s,t\in C$ if $s\leq t$ and $R(s)=R(t)$, then $s=t$.
\een

As in \cite{law1}, the $\Omega$-structured category $(C,\circ,D,R,\leq)$ is said to be  {\em sufficiently regular} if it satisfies the following conditions (together called (OC6) in \cite{law1}): 
\bi
\item[\textup{(OC6a)}] for all $x\in C$ and $e\in D(C)$ for which $e\leq D(x)$, there is $e|x=max\{y\leq x\mid D(y)\leq e\}$, and  $D(e|x)=e$;
\item[\textup{(OC6b)}] for all $x\in C$ and $e\in D(C)$ for which $e\leq R(x)$, there is $x|e=max\{y\leq x\mid R(y)\leq e\}$, and  $R(x|e)=e$.
\ei

In an $\Omega$-structured category $(C,\circ,D,R,\leq)$, if (OC6a) (respectively (OC6b)) is satisfied, $x\in C$ and $e\in D(C)$ with $e\leq D(x)$ (resp. $e\leq R(x)$), we call $e|x$ (resp. $x|e$) the {\em restriction} (resp. {\em corestriction}) of $x$ by $e$.

A further possible law for $\Omega$-structured categories which is considered in \cite{law1} is the following: 
\ben[label=\textup{(OC7)}]  
\item if $a\leq b\circ c$, then there are $b'\leq b,c'\leq c$ such that $a=b'\circ c'$.
\een

A {\em regular} category is a sufficiently regular one satisfying (OC7).  

A further possible law holding in an $\Omega$-structured category $C$ considered in \cite{law1} is the following, which asserts that $D(C)$ is an order-ideal: 
\ben[label=\textup{(OCI)}]  
\item if $a\in C$ and $a\leq e\in D(C)$, then $a\in D(C)$.
\een

As Lawson notes in \cite{law1}, de Barros makes the following definition in \cite{debar}.

\begin{dfn}   \label{treg}
An $\Omega$-structured category satisfying (OC4), (OC6), (OC7) and (OCI) is a {\em totally, regularly ordered category}.
\end{dfn}

A stronger property (strictly, pair of properties) than (OC6) for $\Omega$-structured categories was considered in \cite{law1} and called (OC8) there:
\ben
\item[\textup{(OC8a)}] 
for all $x\in C$ and $e\in D(C)$ such that $e\leq D(x)$, $e|x$ is the unique $y\leq x$ such that $D(y)=e$;
\item[\textup{(OC8b)}]  for all $x\in C$ and $e\in D(C)$ such that $e\leq R(x)$, $x|e$ is the unique $y\leq x$ such that $R(y)=e$.
\een

In \cite{law1}, an $\Omega$-structured category satisfying (OC8a) (resp, (OC8b)) was said to have {\em restrictions} (resp. {\em corestrictions}).  (Strictly speaking, Lawson in \cite{law1} uses terminology dual to that used here, since compositions are read right-to-left there.)  Hence, the notions of restriction and corestriction defined here based on (OC6) are more general than Lawson's.

It was noted in Proposition 2.8 of \cite{law1} that if (OC8) holds in an $\Omega$-structured category, then so do (OC4), (OC7) and (OCI), and of course (OC6) obviously holds also, so the category is totally, regularly ordered.  Examining the short proof of this result in \cite{law1} shows that (OC8a) is sufficient for not just (OC4), (OC7) and (OCI), but also the stronger condition (OC4A); likewise (OC8b) is sufficient for (OC4B).  It is also easy to see that (OC4A) and (OC6) together imply (OC8a); indeed (OC4A) and (OC8a) are equivalent in the presence of (OC6).   We have therefore shown the following.

\begin{pro}   \label{oc}
For an $\Omega$-structured category, (OC8a) is equivalent to (OC4A) and (OC6), and (OC8) is equivalent to (OC4A), (OC4B) and (OC6).
\end{pro}

The Ehresmann categories of \cite{law1}, shown there to be equivalent to (unordered) Ehresmann semigroups, can now be defined.  (Again note that Lawson used dual definitions in \cite{law1}.)

\begin{dfn}   \label{Ehrscat}
A structure $(C,\circ,D,R,\leq_l,\leq_r)$ is said to be an {\em Ehresmann category} if $(C,\circ,D,R)$ is a category, and $\leq_l$ and $\leq_r$ are  partial orders on $C$ such that:
\bi
\item $(C,\circ,D,R,\leq_l)$ is $\Omega$-structured and satisfies (OC8a);
\item $(C,\circ,D,R,\leq_r)$ is $\Omega$-structured and satisfies (OC8b);
\item $\leq_l$ and $\leq_r$ agree on $D(C)$ (call it $\leq$ there);
\item $(D(C),\leq)$ is a meet-semilattice;
\item $\leq_l\circ \leq_r=\leq_r\circ \leq_l$ (so each of these equals the join of $\leq_l,\leq_r$);
\item if $x\leq_r y$ and $e\in D(C)$ then $(D(x)\wedge e)|x\leq_r (D(y)\wedge e)|y$;
\item if $x\leq_l y$ and $e\in D(C)$ then $x|(R(x)\wedge e)\leq_l y|(R(y)\wedge e)$.
\ei
\end{dfn}

This is how the defining laws appeared in \cite{law1}, except that the first law above was stated in terms of $(C,\circ,D,R,\leq_l)$ being an ordered category satisfying (OC8a).  But we have seen that amongst $\Omega$-structured categories, (OC8a) implies (OC4A) and hence (OC4).  So it is sufficient to require that $(C,\circ,D,R,\leq_l)$ be $\Omega$-structured satisfying (OC8a); similarly for the second law.  Note that assuming (OC8a) holds, the restriction of $\leq_l$ to $D(C)$ serves to define $\leq_l$ on all of $C$, since $s\leq_l t$ if and only if $D(s)\leq D(t)$ and $s=D(s)|t$.

Lawson showed in \cite{law1} that an Ehresmann semigroup $(S,\cdot,D,R)$ gives rise to an Ehresmann category $(S,\circ,D,R,\leq_l,\leq_r)$ where $\circ$ is as in Definition \ref{pp} and $\leq_l$, $\leq_r$ are defined via $x\leq_l y$ whenever $x=D(x)y$ and $x=yR(x)$ respectively, and then $e|s=es$ (if $e\leq D(s)$) and $s|f=sf$ (if $f\leq R(s)$).  Conversely, he showed that every Ehresmann category $(C,\circ,D,R,\leq_l,\leq_r)$ gives rise to an Ehresmann semigroup $(C,\otimes,D,R)$ by setting $s\otimes t=(s|(R(s)\wedge D(t))\circ (R(s)\wedge D(t))|t)$.  He showed that these constructions are mutually inverse, and indeed that the categories are isomorphic (where one equips the category of Ehresmann semigroups with Ehresmann semigroup homomorphisms and the category of Ehresmann categories with morphisms that are so-called strongly ordered functors -- functors preserving both the given partial orders as well as meets of elements of $D(C)$).

We now turn to the rather simpler approach used by Lawson in \cite{law2}.
Thus, let $(C,\circ,D,R)$ be a (small) category with $D(C)$ its set of identities.  Lawson defined an {\em Ehresmann biaction} on $C$ as follows: 
\ben[label=\textup{(E\arabic*)}]
\item $D(C)$ is equipped with the structure of a commutative,
idempotent semigroup; denote the operation by $\wedge$ and interpret it as meet (so define $e\leq f$ whenever $e=e\wedge f$).
\item There is a left action $D(C)\times C\rightarrow C$ denoted
by $(e, a) \rightarrow e \cdot a$ (hence satisfying $(e\wedge f)\cdot a=e\cdot (f\cdot a)$) 
such that $D(a)\cdot a = a$ and $D(e \cdot a) = e\wedge D(a)$; there is a right action $C\times D(C)\rightarrow C$ denoted
by $(a, e) \rightarrow a \cdot e$ such that $a\cdot R(a) = a$ and $R(a \cdot e) = R(a)\wedge e$. 
\item The biaction property $(e \cdot  a)\cdot f = e\cdot (a\cdot f)$ holds.
\item For all $e\in D(C)$ and $a\in S$, $e\cdot a = e\wedge a$ and $a\cdot e = a\wedge e$ if $a\in D(C)$.
\item $R(e\cdot a) \leq R(a)$ and $D(a\cdot e)\leq D(a)$ for all $a\in C$ and $e\in D(C)$.
\item When $a\circ b$ exists for some $a,b\in C$, then for all $e\in D(C)$, $e\cdot (a\circ b) = (e\cdot a)\circ(R(e \cdot  a)\cdot b)$ and $(a\circ b)\cdot e = (a \cdot  D(b\cdot e))\circ (b\cdot e)$.
\een

In \cite{law2}, Lawson showed that Ehresmann semigroups give rise to categories with Ehresmann biactions and conversely.  If $(S,\cdot,D,R)$ is an Ehresmann semigroup, turn it into a category by defining the partial products $a\circ b$ as in Definition \ref{pp} above, by defining the left action via $e\cdot a=ea$ and similarly for the right action.  In the other direction, given a category $C$ with Ehresmann biaction, define the ``pseudoproduct" 
$a\bullet b=(a\cdot D(b))\circ (R(a)\cdot b)$, and then $(C,\bullet, D,R)$ is an Ehresmann semigroup; moreover these constructions are mutually inverse (and there is a suitable category isomorphism, as follows from results in \cite{fitzkin}).  We use these facts in what follows.

A more general approach in which the identities $D(C)$ have band structure rather than semilattice structure, is considered in \cite{fitzkin}, and an equivalence between such so-called transcription categories and localisable semigroups is established there.  Indeed the results in the first main section of \cite{law2} can be viewed as a special case of those obtained in \cite{fitzkin}.  For reasons that will become clear shortly, we concentrate on Lawson's approach here.

\section{Ordered Ehresmann semigroups}

\subsection{Basic properties and examples}

In the current work, we consider Ehresmann semigroups endowed with a compatible but not algebraically definable partial order.  In fact we do this in what initially seems to be greater generality. 

\begin{dfn}
Let $S$ be a set.  We say $(S,\cdot,D,R,\leq)$ is an {\em ordered localisable semigroup} if:
\ben[label=\textup{(OS\arabic*)}]
\item $(S,\cdot,D,R)$ is a localisable semigroup and $(S,\leq)$ is a partially ordered set;  \label{OS1}
\item if $a,b\in S$ and $a\leq b$ then $D(a)\leq D(b)$ and $R(a)\leq R(b)$;  \label{OS2}
\item if $a,b\in S$ and $a\leq b$ and $c\leq d$, then $ac\leq bd$ (that is, $(S,\cdot,\leq)$ is an {\em ordered semigroup});  \label{OS3}
\item[\textup{(OS6)}] if $a\in S$ and $e\in D(S)$, then $ae\leq a$ and $ea\leq a$; and \label{OS6}
\item[\textup{(OSI)}] if $a\in S$ and $e\in D(S)$, then $a\leq e$ implies $a\in D(S)$. \label{OSI}
\een
\end{dfn}

(The seemingly odd labelling of the defining laws above is to help make clear how the various properties for the semigroup case correspond to those defined in the category case as in Subsection \ref{ESNEhrs}.)

We might as well have restricted attention to Ehresmann semigroups all along because of the following.

\begin{pro}
Every ordered localisable semigroup is Ehresmann, and then the given order agrees with the semilattice order on $D(S)$: for all $e,f\in D(S)$,  $e\leq f$ if and only if $e=ef$.
\end{pro}
\begin{proof}
Let $(S,\cdot,D,R,\leq)$ be an ordered localisable semigroup.  Then for all $e,f\in D(S)$, $ef=D(ef)ef=efef\leq fef\leq fe$ by \ref{OS6}, so $ef\leq fe$; dually, $fe\leq ef$ and so $ef=fe$.  Hence $S$ is Ehresmann.

If $e=ef$ then $e\leq f$ by (OS6).  Conversely, if $e\leq f$, then by \ref{OS3} and (OS6), $e=ee\leq ef\leq e$, so $e=ef$.
\end{proof}

Hence, from now on we refer only to ordered Ehresmann semigroups.

\begin{dfn}
Let $\leq$ be a partial order on the Ehresmann semigroup $(S,\cdot,D,R)$.  We say it is an {\em Ehresmann order} if $(S,\cdot,D,R,\leq)$ is an ordered Ehresmann semigroup.
\end{dfn}

Ordered Ehresmann semigroups are easily seen to generalise Boolean Ehresmann monoids as defined in \cite{law2}.  Boolean Ehresmann monoids include as examples the Ehresmann semigroup $\Rel(X)$ (for any non-empty set $X$) equipped with the partial order of inclusion, and hence any subsemigroup of $\Rel(X)$ closed under $D$ and $R$ is an ordered Ehresmann semigroup.  Generalising this example is the Boolean Ehresmann monoid $P(C)$ for any category $C$ (see \cite{law2} for its definition); hence any subsemigroup closed under $D$ and $R$ will be ordered Ehresmann.

A given Ehresmann semigroup can be Ehresmann ordered in more than one way.  

\begin{eg} A two-element Ehresmann semigroup with two Ehresmann orders. \label{twoord} \end{eg}

Let $S=\{0,1\}$ be the two-element monoid with zero $0$, and define $D(x)=1=R(x)$ for all $x\in S$.  Then $(S,\cdot,D,R)$ is easily seen to be an Ehresmann semigroup, and it has two non-isomorphic Ehresmann orders, $\leq_1$ and $\leq_2$.  In the first, $1\leq_1 0$.  (This example may be realised as binary relations on $X$ equipped with composition, $D,R$ and inclusion, with $1$ the diagonal relation and $0$ the full relation.)  In the second, $0,1$ are not related, so that $\leq_2$ is equality.  (Verification of the defining laws for Ehresmann orders is routine.)  
\bigskip

On the other hand, we have the following.

\begin{eg} An Ehresmann semigroup with no Ehresmann order. \label{notord} \end{eg}

Consider the subsemigroup $S=\{c,d,P_x,P_y,P_z,1\}$ of the full tranformation semigroup on $X=\{x,y,z\}$, in which $c=\{(x,x),(y,y),(z,x)\}, d=\{(x,x),(y,y),(z,y)\}$, $P_x$ is the projection onto $\{x\}$, and similarly for $P_y,P_z$, and $1$ is the identity map on $X$.  It is easily checked that $S$ is closed under composition, and has the following multiplication table.

\begin{center}
\begin{tabular}{c|cccccc}
$\circ$&$c$&$d$&$P_x$&$P_y$&$P_z$&$1$\\
\hline
$c$&$c$&$c$&$P_x$&$P_y$&$P_z$&$c$\\
$d$&$d$&$d$&$P_x$&$P_y$&$P_z$&$d$\\
$P_x$&$P_x$&$P_x$&$P_x$&$P_y$&$P_z$&$P_x$\\
$P_y$&$P_y$&$P_y$&$P_x$&$P_y$&$P_z$&$P_y$\\
$P_z$&$P_x$&$P_y$&$P_x$&$P_y$&$P_z$&$P_z$\\
$1$&$c$&$d$&$P_x$&$P_y$&$P_z$&$1$
\end{tabular}
\end{center}

Notice that $(S,\cdot)$ is a band.  Define $D(P_x)=D(P_y)=D(P_z)=P_z$ and $D(c)=D(d)=D(1)=1$, with $R(P_x)=R(P_y)=R(c)=R(d)=R(1)=1$ and $R(P_z)=P_z$.  Then $(S,\cdot,D,R)$ is an Ehresmann semigroup with $D(S)=\{P_z,1\}$.  To verify this, note that $D(S)$ is a semilattice, and that all the Ehresmann semigroup laws follow immediately except \ref{L3}.  For this, note that for $s,t\in S$, $D(st)=D(sD(t))$ and $R(ts)=R(R(t)s)$ are trivially true if $t\in D(S)$.   Otherwise, we consider cases.

So consider $D(ct)$ for $t\in \{c,d,P_x,P_y\}$.  Then $D(ct)=D(c)$ if $t=c,d$, but for such $t$, $D(t)=1$, and so $D(ct)=D(cD(t))$; if $t=P_x$ or $t=P_y$, then $D(ct)=D(t)=P_z$, whereas $D(cD(t))=D(cP_z)=D(P_z)=P_z=D(ct)$.  So in all cases, $D(ct)=D(cD(t))$.  The argument for $D(dt)=D(dD(t))$ is very similar.  From the table we see that $D(P_xt)=P_z$ for all $t$, so $D(P_xt)=D(P_xD(t))=P_z$.  Similarly for $P_y,P_z$.  And $D(1t)=D(t)=D(1D(t))$.  So $D(st)=D(sD(t))$ holds for all $s,t\in S$.

Next, consider $R(tc)$, $t\in \{c,d,P_x,P_y\}$.  If $t=c,d$ then $R(tc)=1=R(c)=R(1c)=R(R(t)c)$; if $t=P_x$ or $t=P_y$ then $R(tc)=R(t)=P_z=R(P_zc)=R(R(t)c)$, so in all cases $R(tc)=R(R(t)c)$.  Again, the argument that $R(td)=R(R(t)d)$ is very similar.  Finally, $R(tP_x)=R(P_x)$ for all $t$, so in particular, $R(tP_x)=R(R(t)P_x)$ for all $t$, and similarly for $R(tP_y)$ and $R(tP_z)$.  Finally, $R(t1)=R(R(t)1)$.  So in all cases, $R(ts)=R(R(t)s)$ for all $t,s\in S$.

This completes the proof that $(S,\cdot,D,R)$ is an Ehresmann semigroup.  Suppose the partial order $\leq$ on $S$ is an Ehresmann order.  Now $P_zc\leq c$ and so $P_y(P_zc)\leq P_yc$, and similarly $P_xP_zd\leq P_xd$.  Direct computation then gives that $P_x=P_yP_zc\leq P_yc=P_y$, but also $P_y=P_xP_zd\leq P_xd=P_x$, and so $P_x=P_y$, a contradiction.

\begin{dfn}
Let $(S,\cdot,D,R,\leq)$ and $(T,\cdot,D,R,\leq)$ be ordered Ehresmann semigroups.  Then $F:S\rightarrow T$ is an {\em ordered Ehresmann semigroup homomorphism} if it is an Ehresmann semigroup homomorphism that respects the Ehresmann order: for all $s,t\in S$, $s\leq t$ implies $sF\leq tF$. 
\end{dfn}

It is easily seen that the class of ordered Ehresmann semigroups forms a category in which the morphisms are ordered Ehresmann semigroup homomorphisms.   (One need only verify closure under composition).  Moreover, using the terminology of \cite{fitzkin}, these are nothing but order-preserving localisable semigroup morphisms.

An ordered Ehresmann semigroup $(S,\cdot,D,R,\leq)$ can satisfy semigroup versions of the other properties previously considered for categories: 
\ben
\item[\textup{(OS4)}] for all $s,t\in S$, if $s\leq t, D(s)=D(t)$ and $R(s)=R(t)$, then $s=t$;
\item[\textup{(OS4A)}] for all $s,t\in S$ if $s\leq t$ and $D(s)=D(t)$, then $s=t$;
\item[\textup{(OS4B)}] for all $s,t\in S$ if $s\leq t$ and $R(s)=R(t)$, then $s=t$;
\item[\textup{(OS7)}] for all $s,t\in S$, if $u\leq st$ then there are $s'\leq s, t'\leq t$ for which $u=s't'$. 
\een

\subsection{When the order is algebraically definable}

Let $(S,\cdot,D,R)$ be an Ehresmann semigroup.  Following \cite{law1}, denote by $\leq_l$ the partial order given by $s\leq_l t$ if and only if $s=D(s)t$, and define $\leq_r$ dually, via $s\leq_r t$ if and only if $s=tR(s)$.  It is noted in \cite{law1} and easy to see that these two partial orders permute, and hence their join is simply their composite as binary relations; define $\leq_e$ to be this join.  As noted in \cite{law1}, $\leq_e$ is a partial order, agreeing with the natural partial order on the (meet-)semilattice $D(S)$  (Proposition 3.11 applied to an Ehresmann semigroup), and $s\leq_e t$ if and only if $s=gth$ for some $g,h\in D(S)$ (Proposition 3.13 in \cite{law1}).  
 
\begin{lem}   \label{leqe0}
Suppose $(S,\cdot,D,R)$ is an Ehresmann semigroup.  Then for all $s,t\in S$, $s\leq_e t$ if and only if $s=D(s)tR(s)$.  Hence any Ehresmann order $\leq$ on $S$ contains $\leq_e$ (in the sense that any pair related by $\leq_e$ is also related by $\leq$).
\end{lem}
\begin{proof}
Suppose $s,t\in S$ satisfy $s\leq_e t$, so that $s=gth$ for some $g,h\in D(S)$.  Then $D(s)=D(gth)=D(gD(th))=gD(th)$, so $D(s)\leq g$, and similarly, $R(s)\leq h$.  Hence 
$s=D(s)sR(s)=D(s)gthR(s)=D(s)tR(s)$.  The converse direction is obvious.

Hence, if $\leq$ is an Ehresmann order on $S$, and $s\leq_e t$ for some $s,t\in S$, then $s=D(s)tR(s)\leq tR(s)\leq t$, so $s\leq t$.
\end{proof}

The partial order $\leq_e$ is important partly because in many important types of Ehresmann semigroup, it is an Ehresmann order.  In fact it is always close to being an Ehresmann order, as the following shows.

\begin{pro}  \label{leqe}
Suppose $(S,\cdot,D,R)$ is an Ehresmann semigroup.  Then $\leq_e$ satisfies the laws \ref{OS1}, \ref{OS2}, (OS6) and (OSI) of Ehresmann orders, but may not satisfy \ref{OS3}.
\end{pro}
\begin{proof}
That $\leq_e$ is a partial order has already been noted, so (OS1) is satisfied.  Suppose $a,b\in S$ and $e\in D(S)$.  If $a\leq_e b$ then $a=D(a)bR(a)$, so $D(a)=D(D(a)bR(a))=D(a)D(bR(a))\leq D(bR(a))\leq D(b)$ since $D(b)bR(a)=bR(a)$, and similarly $R(a)\leq R(b)$; hence \ref{OS2} holds.  Now $ae=D(ae)aeR(ae)=D(ae)aeR(R(a)e)=D(ae)aeR(a)$, so $ae\leq_e a$; similarly, $ea\leq_e a$; hence (OS6) holds.  Finally, if $a\leq e$ then $a=D(a)eR(a)\in D(S)$; hence (OSI) holds.  

It follows that if $\leq_e$ satisfies \ref{OS3}, then $(S,\cdot,D,R,\leq_e)$ is an ordered Ehresmann semigroup.  But Example \ref{notord} shows that not every Ehresmann semigroup can be given an Ehresmann order.  Hence \ref{OS3} need not be satisfied by $\leq_e$ in general.
\end{proof}

In \cite{law1}, an Ehresmann semigroup $(S,\cdot,D,R)$ is said to be {\em de Barros} if it is such that $\leq_e$ satisfies \ref{OS3}; from Proposition \ref{leqe}, this is equivalent to requiring that $\leq_e$ be an Ehresmann order on $S$.  

We have the following easy consequence of Lemma \ref{leqe0}.

\begin{cor}
In a de Barros Ehresmann semigroup, $\leq_e$ is the smallest Ehresmann order on $S$.
\end{cor}

A {\em left restriction semigroup with range} $(S,\cdot,D,R)$ will here mean an Ehresmann semigroup additionally satisfying the law $sD(t)=D(st)s$; this is nothing but a left restriction semigroup equipped with a range operation satisfying the Ehresmann semigroup laws for range.  In \cite{law1}, these are referred to as Ehresmann semigroups satisfying the condition $IC_l$.  The main example is perhaps $\PT(X)$, the subsemigroup of $\Rel(X)$ consisting of partial transformations.  This example also satisfies the law $st=su\Rightarrow R(s)t=R(s)u$, and conversely, it was shown in \cite{scheinDR} that any left restriction semigroup with range satisfying this further law admits a functional representation; accordingly, any left restriction semigroup with range satisfying this further law will be called {\em functional}.  

There is an obvious dual definition of {\em right restriction semigroup with domain}; in \cite{law1}, these are the Ehresmann semigroups satisfying $IC_r$.  A {\em restriction semigroup} is an Ehresmann semigroup $(S,\cdot,D,R)$ such that $(S,\cdot,D)$ is a left restriction semigroup and $(S,\cdot,R)$ is a right restriction semigroup; equivalently, $(S,\cdot,D,R)$ is both a left restriction semigroup with range and a right restriction semigroup with domain.  As is well-known, every inverse semigroup is a restriction semigroup if one defines $D(s)=ss'$ and $R(s)=s's$ for all $s$ (where $s'$ is the inverse of $s$).  

It is well-known that in a left restriction semigroup $(S,\cdot,D)$ (possibly without range), $\leq_l$ satisfies \ref{OS3} and moreover that $\leq_e$ (which still makes sense in this one-sided case) equals $\leq_l$.  Hence a left restriction semigroup with range is de Barros; hence so is a right restriction semigroup with domain.  In particular, every restriction semigroup $(S,\cdot,D,R)$ is de Barros, a fact noted in \cite{law1}. 

It follows easily from  Proposition \ref{leqe0} that the class of de Barros semigroups may be viewed as a quasivariety of Ehresmann semigroups, but in fact it is a variety.

\begin{pro}  \label{deBeq}
An Ehresmann semigroup $(S,\cdot,D,R)$ is a de Barros semigroup if and only if for all $s,t\in S$ and $e\in D(S)$, $set=D(set)stR(set)$.
\end{pro}
\begin{proof}
If $(S,\cdot,D,R)$ is de Barros, then $(S,\cdot,D,R,\leq_e)$ is an ordered Ehresmann semigroup, so for all $s,t\in S$ and $e\in D(S)$, $se\leq_e s$ and so $(se)t\leq_e st$, so by Proposition \ref{leqe0}, $set=D(set)stR(set)$.  Conversely, if $(S,\cdot,D,R)$ satisfies $set=D(set)stR(set)$ for all $s,t\in S$ and $e\in D(S)$, then $set\leq_e st$ for all such $s,t,e$.  Suppose $x_1\leq_e y_1, x_2\leq_e y_2$ for some $x_1,x_2,y_1,y_2\in S$.  Then $x_1=f_1y_1g_1$ and $x_2=f_2y_2g_2$ for some $f_1,f_2,g_1,g_2\in D(S)$.  But $g_1f_2\in D(S)$, so $y_1g_1f_2y_2\leq_e y_1y_2$ and hence $y_1g_1f_2y_2=h_1y_1y_2h_2$ for some $h_1h_2\in D(S)$.  Hence $x_1x_2=f_1y_1g_1f_2y_2g_2=f_1h_1y_1y_2h_2g_2$ so because $f_1h_1,h_2g_2\in D(S)$, $x_1x_2\leq_e y_1y_2$.  So $\leq_e$ satisfies \ref{OS3} and so $(S,\cdot,D,R)$ is de Barros.
\end{proof}

The de Barros property can be related to property (OS4) given earlier.

\begin{pro}  \label{deBarros}
Let $(S,\cdot,D,R)$ be an Ehresmann semigroup.   If it is de Barros then the ordered Ehresmann semigroup $(S,\cdot,D,R,\leq_e)$ satisfies (OS4).  Conversely, if $(S,\cdot,D,R,\leq)$ is an ordered Ehresmann semigroup satisfying (OS4), then $(S,\cdot,D,R)$ is de Barros and $\leq$ is $\leq_e$.
\end{pro}
\begin{proof}
Suppose $(S,\cdot,D,R)$ is de Barros.  Suppose $a,b\in S$ are such that $a\leq_e b$, $D(a)=D(b)$ and $R(a)=R(b)$.  Then $a=D(a)bR(a)=D(b)bR(b)=b$.  So the ordered Ehresmann semigroup $(S,\cdot,D,R,\leq_e)$ satisfies (OS4).

Conversely, suppose $(S,\cdot,D,R,\leq)$ is an ordered Ehresmann semigroup satisfying (OS4).  Suppose $a,b\in S$.  If $a\leq b$ then $a=D(a)aR(a)\leq D(a)bR(a)\leq b$, so $$D(D(a)bR(a))=D(D(a)D(bR(a)))=D(a)D(bR(a))\leq D(a)\leq D(D(a)bR(a)),$$ whence $D(a)=D(D(a)bR(a))$; similarly $R(a)=R(D(a)bR(a))$, so by (OS4), $a=D(a)bR(a)$, and so $a\leq_e b$.  If $a\leq_e b$ then $a=D(a)bR(a)\leq b$.  So $\leq$ and $\leq_e$ agree.  Since $\leq$ satisfies \ref{OS3}, so must $\leq_e$, so $(S,\cdot,D,R)$ is de Barros.
\end{proof}

\begin{pro}  \label{arrows}
Let $(S,\cdot,D,R)$ and $(T,\cdot,D,R)$ be de Barros semigroups, with $F:S\rightarrow T$ an Ehresmann semigroup homomorphism.  Then $F$ is an ordered Ehresmann semigroup homomorphism from $(S,\cdot,D,R,\leq_e)$ to $(T,\cdot,D,R,\leq_e)$.
\end{pro}
\begin{proof}
If $s\leq_e t$ then $s=D(s)tR(s)$, so $$sF=(D(s)tR(s))F=(D(s)F)(tF)(R(s)F)=D(sF)(tF)R(sF),$$ 
and so $sF\leq_e tF$.
\end{proof}

By the above result and Proposition \ref{deBarros}, the categories of de Barros semigroups and ordered Ehresmann semigroups satisfying (OS4) are isomorphic.

\begin{pro}  \label{OS47}
For ordered Ehresmann semigroups, (OS4) implies (OS7), but the converse is false.
\end{pro}
\begin{proof}
Suppose $(S,\cdot,D,R,\leq)$ is an ordered Ehresmann semigroup satisfying (OS4); then $\leq$ is $\leq_e$ by Proposition \ref{deBarros}.  Suppose $u\leq st$ for some $u,s,t\in S$.  Then $u\leq_e st$, so $u=D(u)stR(u)$ by Proposition \ref{leqe0}, so letting $s'=D(u)s\leq s$ and $t'=tR(u)\leq t$, clearly $u=s't'$.  Hence (OS7) holds.

Let $S=\{0,1\}$ be the two-element monoid as in Example \ref{twoord}, equipped with the Ehresmann order $\leq_1$ as in that example. Then $(S,\cdot,D,R,\leq)$ satisfies (OS7), as an easy case analysis shows.  However,  $D(1)0R(1)=0\neq 1$, so $1\not\leq_e 0$ by Proposition \ref{leqe0}, and so $\leq$ and $\leq_e$ are distinct, so $(S,\cdot,D,R,\leq)$ does not satisfy (OS4) by Proposition \ref{deBarros}.
\end{proof}
  
Recall that one and two-sided restriction semigroups are de Barros semigroups. There are specialisations of Proposition \ref{deBarros} applying to them. 

\begin{pro}  \label{rest}
Let $(S,\cdot,D,R)$ be an Ehresmann semigroup.  

If $(S,\cdot,D,R)$ is a left restriction semigroup with range (resp. right restriction semigroup with domain), then the Ehresmann order $\leq_e$ on $(S,\cdot,D,R)$ satisfies (OS4A) (resp. (OS4B)).  Conversely, if $(S,\cdot,D,R,\leq)$ is an ordered Ehresmann semigroup satisfying (OS4A) (resp. (OS4B)), then it is a left restriction semigroup with range (resp. right restriction semigroup with domain) in which $\leq$ equals $\leq_e$.  

Hence if $(S,\cdot,D,R)$ is a restriction semigroup, then $(S,\cdot,D,R,\leq_e)$ is an ordered Ehresmann semigroup satisfying (OS4A) and (OS4B).  Conversely, if $(S,\cdot,D,R,\leq)$ is an ordered Ehresmann semigroup satisfying (OS4A) and (OS4B), then it is a restriction semigroup in which $\leq$ equals $\leq_e$.
\end{pro}
\begin{proof}  
Suppose $(S,\cdot,D,R)$ is a left restriction semigroup with range.  As noted prior to Proposition \ref{deBeq}, $S$ is de Barros and so $\leq_e$ is an Ehresmann order on $(S,\cdot,D,R)$. It remains to check (OS4A).  But if $s\leq_e t$ and $D(s)=D(t)$, then $s=D(s)t=D(t)t=t$.

Conversely, suppose $(S,\cdot,D,R,\leq)$ is an ordered Ehresmann semigroup satisfying (OS4A).  Then it satisfies (OS4), and so by Proposition \ref{deBarros}, it is de Barros, and $\leq$ is $\leq_e$.  Moreover, for all $x,y\in S$, $xD(y)=D(xD(y))xD(y)=D(xy)xD(y)\leq D(xy)x$, and 
$$D(xD(y))=D(xy)=D(D(x)xy)=D(D(x)D(xy))=D(D(xy)D(x))=D(D(xy)x),$$ 
so by (OS4A), $xD(y)=D(xy)x$.  Hence, $(S,\cdot,D,R)$ is a left restriction semigroup with range.  The dual case now follows also, and the second claim follows by combining the two one-sided cases.
\end{proof}

\section{An ESN-style theorem for ordered Ehresmann semigroups}

\subsection{Motivation}

One approach to obtaining an ESN-style theorem for ordered Ehresmann semigroups would be to be guided by \cite{law2}, simply building in order information to a category with Ehresmann biaction.  However, we are mainly interested in a second more elegant approach, which although necessarily equivalent to the one just outlined, is more along the lines of \cite{law1}. 

Let $(S,\cdot,D,R)$ be an Ehresmann semigroup with $e,f\in D(S)$, and $s\in S$.  The approach used in \cite{law1} rested on the observation that if $e\leq D(s)$ and $f\leq R(s)$ then $es$ is the unique $t\in S$ for which $t\leq_l s$ and $D(t)=e$, and dually for $sf$.  However, if $\leq$ is an Ehresmann order on $S$, then $es$ is the {\em largest} $t\in S$ (with respect to $\leq$) for which $t\leq s$ and $D(t)\leq e$, and then $D(t)=e$; dually for $sf$. Therefore, part of our approach is to replace (OC8a) and (OC8b) applying to $\leq_l$ and $\leq_r$ respectively by (OC6) applying to $\leq$.  A relatively small number of order-theoretic properties of the category with partial order $(S,\cdot,D,R,\leq)$ then suffice to characterise the partially ordered categories that arise in this way from ordered Ehresmann semigroups.  Remarkably, this approach gives simpler laws than those of Ehresmann categories which model {\em unordered} Ehresmann semigroups, but it is also simpler than an approach in which one enriches the structure of a category with Ehresmann biaction. 

\subsection{Ehresmann-ordered categories}

We now turn to the category-theoretic analogs of ordered Ehresmann semigroups we shall work with here.   To do this, we must consider a further possible property of an $\Omega$-structured category $(C,\circ,D,R,\leq)$, which is a weakening of (OC7):
\ben
\item[\textup{(OC7')}] if $a\leq b\circ c$, then there are $b'\leq b,c'\leq c$ such that $D(a)=D(b'), R(a)=R(c')$, $b'\circ c'$ exists, and $a\leq b'\circ c'$.
\een

If an $\Omega$-structured category satisfies (OC4) and (OC7') then it is easy to see that it satisfies (OC7) as well.  In the general case, we are interested in cases satisfying (OC7') but not necessarily (OC7), hence perhaps not (OC4) either.  

\begin{dfn}  \label{Ehrsordcat}
We say an $\Omega$-structured category $(C,\circ,D,R,\leq)$ is an {\em Ehresmann-ordered category} if:  
\bi
\item it satisfies (OC6), (OC7') and (OCI);
\item $(D(C),\leq)$ is a meet-semilattice.
\ei
\end{dfn}

This definition parallels fairly closely that of ordered Ehresmann semigroups, with only (OC7') and the meet-semilattice assumption having no parallels (although the latter parallels one of the properties of Ehresmann semigroups).  We contrast this relatively simple definition with that of Ehresmann categories as in Definition \ref{Ehrscat}.  Here, there is only one partial order and the first two laws for Ehresmann categories are replaced by the first law above, and only the meet-semilattice property of $(D(C),\leq)$ is retained from all the other laws for Ehresmann categories.

In fact, every ordered Ehresmann semigroup yields an Ehresmann-ordered category.

\begin{pro}  \label{corresp0}
If $(S,\cdot,D,R,\leq)$ is an ordered Ehresmann semigroup, then defining $\circ$ as in Definition \ref{pp} makes $\C(S)=(S,\circ,D,R,\leq)$ an Ehresmann-ordered category in which for all $a\in S$, if $e\leq D(a)$ then $e|a=ea$, and if $e\leq R(a)$ then $a|e=ae$.
\end{pro}
\begin{proof}
Suppose $(S,\cdot,D,R)$ is an Ehresmann semigroup.  That $(S,\circ,D,R)$ is a category has been discussed already.  Suppose $\leq$ is an Ehresmann order on $(S,\cdot,D,R)$.  Only (OC6) and (OC7') are not obvious in showing that $(S,\circ,D,R,\leq)$ is an Ehresmann-ordered category.  

If $e\leq D(a)$, then $e|a=ea\leq a$ and $D(e|a)=D(ea)=D(eD(a))=D(e)=e$, and if $y\leq a$ with $D(y)\leq e\leq D(a)$, then $y=D(y)y\leq ea=e|a$.   Dually for $a|e$ if $e\leq R(a)$.  So (OC6) holds.  

For (OC7'), suppose $z\leq x\circ y$, so $D(z)\leq D(x\circ y)=D(x)$.  Let $x'=D(z)x\leq x$ and $y'=R(D(z)x)y\leq y$.  Then because  $R(D(z)x)\leq R(x)=D(y)$, we obtain $$D(y')=D(R(D(z)x)y)=R(D(z)x)D(y)=R(D(z)x)=R(x'),$$ so $x'\circ y'$ exists and $x'\circ y'\leq x\circ y$.  Moreover $$z=D(z)z\leq D(z)xy=D(z)xR(D(z)x)y=x'\circ y',$$ with $$D(x'\circ y')=D(x')=D(D(z)x)=D(D(z)D(x))=D(D(z))=D(z).$$   Applying the dual argument to $z\leq x'\circ y'$ gives $x''<x'<x, y''<y'<y$ such that $z\leq x''\circ y''$ and $R(y'')=R(z)$, whilst $D(z)\leq D(x'')\leq D(x')=D(z)$, so $D(z)=D(x'')$, as required.
\end{proof}

Recall that a morphism between ordered Ehresmann semigroups is an order-preserving semigroup homomorphim that respects $D$ and $R$.  

\begin{dfn}  \label{ecmorph} 
A {\em morphism} between Ehresmann-ordered categories $(C_1,\circ,D,R,\leq)$ and $(C_2,\circ,D,R,\leq)$ is a functor $F:(C_1,\circ,D,R)\rightarrow (C_2,\circ,D,R)$ satisfying the following:
\bi
\item if $s,t\in C_1$ and $s\leq t$ then $sF\leq tF$;
\item if $e,f\in D(C_1)$ then $(e\wedge f)F=eF\wedge fF$;
\item if $s\in C_1$ and $e\leq D(s)$ then $(e|s)F=(eF)|(sF)$, and dually if $f\leq R(s)$ then $(s|f)F=(sF)|(fF)$.
\ei
\end{dfn}

The definition used by Lawson in \cite{law1} omitted the final requirement given above, but as shown there in the proof of Lemma 4.23, it follows from the first two requirements, at least in that setting, and the argument carries over easily to the current setting if (OC8) is assumed.  In fact (OC8) is necessary for this simplification, as the following example shows.  

Let $X$ be a set with at least two elements, with $S=\{0,1,\nabla\}\subseteq \Rel(X)$, where $0$ is the empty relation, $1$ is the diagonal relation (identity function on $X$) and $\nabla$ is the full relation on $X$.  Evidently, $S$ is an ordered Ehresmann subsemigroup of $(\Rel(X),\cdot,D,R,\subseteq)$.  Define $F: S\rightarrow S$ as follows: $0F=1F=1, \nabla F=\nabla$.  This is routinely checked to be order-preserving, to preserve all category products on $\C(S)$, to respect $D$ and $R$, and to preserve all meets of members of $D(S)$.  So, viewed as a functor $F: \C(S)\rightarrow \C(S)$ on the Ehresmann-ordered category $\C(S)$ (upon using Proposition \ref{corresp0}), it satisfies the first two conditions of a morphism.  But note that $0\leq 1=D(\nabla)$, and $0|\nabla=0$, so $(0|\nabla)F=0F=1$, whereas $(0F)|(\nabla F)=1|\nabla=\nabla$, so the third is not satisfied.  (It follows that (OC8) is not satisfied, which can be seen directly because both $1$ and $\nabla$ are below $\nabla$ even though both have domain $1$.)

Consider the Ehresmann semigroup $S$ as in Example \ref{twoord}, along with the two Ehresmann orders on it, $\leq_1$ (in which $1\leq_1 0$) and $\leq_2$ (the equality relation).  Then we know that $(S,\circ,D,R,\leq_1)$ and $(S,\circ,D,R,\leq_2)$ are Ehresmann-ordered categories by Proposition \ref{corresp0}.  The identity mapping $I: S\rightarrow S$ evidently satisfies the second and third conditions in Definition \ref{ecmorph}, but fails to satisfy the first, since $1\leq_1 0$ yet $1\not\leq_2 0$.  So the first condition in Definition \ref{ecmorph} cannot in general be inferred from the second and third either.

Clearly morphisms between Ehresmann-ordered categories are closed under composition, and all identity mappings on Ehresmann-ordered categories are morphisms, and we define the category of Ehresmann-ordered categories to have as its arrows the morphisms between Ehresmann-ordered categories.  

\begin{thm}  \label{correspa}
Suppose $(C,\circ,D,R,\leq)$ is an Ehresmann-ordered category.  Then it has Ehresmann biaction in which $D(C)$ is ordered by $\leq$, and for all $e\in D(C)$ and $x\in S$,
$$e\cdot x=(e\wedge D(x))|x, x\cdot e=x|(R(x)\wedge e).$$
\end{thm}
\begin{proof}
First, note that $e\cdot x$ exists, and $e\cdot x\leq x$ and $D(e\cdot x)\leq e$.   Moreover if $y$ is such that $y\leq x, D(y)\leq e$, then because $D(y)\leq D(x)$, we have that $D(y)\leq e\wedge D(x)$, so $y\leq (e\wedge D(x))|x=e\cdot x$.  Trivially $D(e\cdot x)=e\wedge D(x)$, so $e\cdot x$ is the largest $y\leq x$ for which $D(y)\leq e$, and then $D(e\cdot x)=e\wedge D(x)$.  Dually for $x\cdot e$.

Now if $e,f\in D(C)$ and $a\in C$, then we must show that $(e\wedge f)\cdot a=e\cdot(f\cdot a)$.  Suppose $y\leq (e\wedge f)\cdot a$.  Then $y\leq a$, $D(y)\leq D(a)\wedge e\wedge f\leq D(a)\wedge f$ and so $y\leq f\cdot a$, and since $D(y)\leq D(a)\wedge e\wedge f=D(f\cdot a)\wedge e\leq e$, we have that $y\leq e\cdot(f\cdot a)$.  Conversely, if $y\leq e\cdot(f\cdot a)$ then $y\leq f\cdot a$ and $D(y)\leq D(f\cdot a)\wedge f=D(a)\wedge e\wedge f$, so $y\leq (e\wedge f)\cdot a$.   Hence $e\cdot(f\cdot a)=(e\wedge f)\cdot a$. Dualise to get that $D(C)$ acts on the right using corestriction.  The rest of (E1) is satisfied by Propostion \ref{rest}.  

For $e\in D(C)$ and $a\in C$, $D(e\cdot a)=D(e)\wedge a$ by Propostion \ref{rest}, and clearly $D(a)\cdot a=a$.  Dualising gives (E2).  

For (E3), we must show that $e\cdot(a\cdot f)=(e\cdot a)\cdot f$ for all $a\in C$ and $e,f\in D(C)$.  But for $a,y\in C$ and $e,f\in D(C)$, the following are equivalent: $y\leq (e\cdot a)\cdot f$; $y\leq e\cdot a$ and $R(y)\leq f$; $y\leq a$, $D(y)\leq e$ and $R(y)\leq f$; $y\leq a\cdot f$ and $D(y)\leq e$; $y\leq e\cdot(a\cdot f)$.

For (E4), Proposition \ref{rest} (and its dual) again gives the desired results, and (E5) follows from monotonicity of $D$ and $R$.  

For (E6), we must show that $e\cdot(x\circ y)=(e\cdot x)\circ (R(e\cdot x)\cdot y)$,and dually.  Suppose $z\leq e\cdot(x\circ y)$.  Then $z\leq x\circ y$ and $D(z)\leq e$.  By assumption, there are $x'\leq x,y'\leq y$ for which $z\leq x'\circ y'$ and $D(x')=D(z)\leq e$, so $x'\leq e\cdot x$, and so $D(y')=R(x')\leq R(e\cdot x)$, so $y'\leq R(e\cdot x)\cdot y$.  So $z\leq x'\circ y'\leq (e\cdot x)\circ (R(e\cdot x)\cdot y)$ (the latter existing since $D(R(e\cdot x)\cdot y)=R(e\cdot x)\wedge D(y)=R(e\cdot x)$ since $R(e\cdot x)\leq R(x)=D(y)$).  So $e\cdot(x\circ y)\leq (e\cdot x)\circ (R(e\cdot x)\cdot y)$.  Conversely, it is evident that $(e\cdot x)\circ R(e\cdot x)\cdot y\leq x\circ y$, and because $D((e\cdot x)\circ R(e\cdot x)\cdot y)=D(e\cdot x)=e\wedge D(x)\leq e$, we have that $(e\cdot x)\circ R(e\cdot x)\cdot y\leq e\cdot(x\circ y)$.  The dual case follows similarly.
\end{proof}

A notion of functor between transcription categories was defined in \cite{fitzkin}, and therefore applies to categories with Ehresmann biaction.  Thus, $F:C_1\rightarrow C_2$ is such a functor providing it is a category functor additionally respecting the biaction, so that $(e\cdot s)F=eF\cdot sF$ and dually, for any $e\in D(C_1)$ and $s\in C_1$.
 
\begin{pro}  \label{istransfunc}
Let $(C_1,\circ,D,R,\leq)$ and $(C_2,\circ,D,R,\leq)$ be Ehresmann-ordered categories.  If $F$ is an Ehresmann-ordered category morphism from the first to the second, then it satisfies $(e\cdot s)F=(eF)\cdot (sF)$ for all $e\in D(C_1)$ and $s\in C_1$.
\end{pro}
\begin{proof}
For all $e\in D(C_1)$ and $s\in C_1$,
\bea
(e\cdot s)F&=&((e\wedge D(s))|s)F\\
&=&((e\wedge D(s))F)|(sF)\\
&=&(eF\wedge D(s)F)|(sF)\\
&=&(eF\wedge D(sF))|(sF)\\
&=&(eF)\cdot (sF),
\eea
as required.
\end{proof}

Hence an Ehresmann-ordered category morphism is a functor between the induced categories with Ehresmann biaction (given by Theorem \ref{correspa}).

\subsection{The ESN-style Theorem}

First, we observe that there is a converse to Proposition \ref{corresp0}.

\begin{pro}  \label{corresp1}
If $(C,\circ,D,R,\leq)$ is an Ehresmann-ordered category, then setting $$s\otimes t=s|(R(s)\wedge D(t))\circ (R(s)\wedge D(t))|t,$$ 
$(C,\otimes, D,R,\leq)$ is an ordered Ehresmann semigroup.  
\end{pro}
\begin{proof}
Suppose $C$ is an Ehresmann-ordered category.  Then it induces an Ehresmann biaction by Theorem \ref{correspa}, and hence by Theorem 2.7 of \cite{law2} can be made into an Ehresmamm semigroup $\Se(C)$ under $\otimes$ as given.  Equip $\Se(C)$ with the partial order of $C$.  Most of the properties of an Ehresmann order are easily seen, the only non-obvious one being
$s_1\leq s_2, t_1\leq t_2\Rightarrow s_1s_2\leq t_1t_2$.  But $t_1t_2=t_1\otimes t_2=(t_1| e)\circ (e|t_2)\leq (s_1|e)\circ (e|s_2)$ (where $e=R(t_1)\wedge D(t_2)$, noting that $e\leq R(s_1),D(s_2)$), which in turn is below $s_1\otimes s_2=(s_1\cdot f)\circ (f\cdot s_2)$ where $e\leq f=R(s_1)\wedge D(s_2)$.
\end{proof}

\begin{thm}  \label{corresp}
The two constructions given in Theorem \ref{corresp0} and \ref{corresp1} are mutually inverse, and indeed there is a category isomorphism between the category of ordered Ehresmann semigroups and the category of Ehresmann-ordered categories.
\end{thm}
\begin{proof}
Viewed as Ehresmann semigroups and categories with Ehresmann biaction, these constructions are precisely those shown to be mutually inverse in \cite{law2} (and indeed in \cite{fitzkin}); since the order does not change under the construction, this carries over to when they are viewed as Ehresmann-ordered categories and ordered Ehresmann semigroups respectively.  To prove the category isomorphism, it therefore suffices to prove that for any two ordered Ehresmann semigroups $S,T$, a function $F:S\rightarrow T$ is an ordered Ehresmann semigroup homomorphism if and only if it is an Ehresmann-ordered category morphism $\C(S)\rightarrow \C(T)$.  If $F:S\rightarrow T$ is an ordered Ehresmann semigroup homomorphism, then it is certainly a functor $\C(S)\rightarrow \C(T)$, moreover one that is order-preserving (since the order has not changed) and meet-preserving in $D(\C(S))=D(S)$ (since meet in $D(\C(S))$ is just multiplication in $D(S)$), and so it is an Ehresmann-ordered category functor.  Conversely, if $F:\C(S)\rightarrow \C(T)$ is an Ehresmann-ordered functor, then by Proposition \ref{istransfunc}, it is a transcription category morphism in the sense of \cite{fitzkin}, and so by Theorem 4.8 there, it is an Ehresmann semigroup homomorphism, moreover one preserving the order, and hence is an ordered Ehresmann sermigroup homomorphism $S\rightarrow T$.  
\end{proof}  

If $(S,\cdot,D,R,\leq)$ is an ordered Ehresmann semigroup, then for all $s,t\in S$, $s\leq t$ if and only if $D(s)\leq D(t)$, $R(s)\leq R(t)$, and $s\leq D(s)tR(s)$, as is easily seen.  In the corresponding Ehresmann-ordered category $(S,\circ,D,R,\leq)$ (using Theorem \ref{corresp}), we have $D(s)tR(s)=D(s)\cdot t\cdot R(s)$, using the biaction defined in Theorem \ref{correspa}.  Note further that $D(s)=D(D(s)sR(s))\leq D(D(s)tR(s))=D(s)D(tR(s))\leq D(s)$, so that $D(D(s)tR(s))=D(s)$, and similarly $R(D(s)tR(s))=R(s)$.  All of this shows that $\leq$ on all of $S$ may be recovered from knowledge of the biaction plus the restriction of $\leq$ to the hom-sets of $S$ viewed as a category.  There will therefore be a characterisation of Ehresmann-ordered categories as categories with Ehresmann biaction such that the hom-sets are all partially ordered and satisfy a handful of further properties that enforce precisely the Ehresmann-ordered category laws.  
We do not pursue the precise details further here.

\subsection{Some special cases} 

We have the following specialisations of Theorem \ref{corresp}.

\begin{thm}  \label{corresp4}
There is a category isomorphism between the category of ordered Ehresmann semigroups satisfying (OS4) and the category of Ehresmann-ordered categories satisfying (OC4).

Similarly, there is a category isomorphism between the category of ordered Ehresmann semigroups satisfying (OS7) and the category of Ehresmann-ordered categories satisfying (OC7).
\end{thm}
\begin{proof}
The laws (OS4) and (OC4) have identical forms and do not depend on the product (semigroup or category-theoretic), eatablishing the first claim.

Suppose the ordered Ehresmann semigroup $(S,\cdot,D,R,\leq)$ satisfies (OS7).  Then suppose that in the Ehresmann-ordered category $(S,\circ,\leq)$, $u\leq s\circ t$.  Then $u\leq st$, so there are $s'\leq s,t'\leq t$ such that $u=s't'$.  Let $s''=s'D(t'),t''=R(s')t'$; then $s''t''=s't'=u$ and 
$$R(s'')=R(s'D(t'))=R(R(s')D(t'))=R(s')D(t')=D(R(s')D(t'))=D(R(s')t')=D(t''),$$ 
so $u=s''\circ t''$.  So the Ehresmann-ordered category $(S,\circ,\leq)$ satisfies (OC7).

Conversely, suppose the Ehresmann-ordered category $(C,\circ,D,R,\leq)$ satisfies (OC7).  Further suppose that in the ordered Ehresmann semigroup $(C,\otimes,D,R,\leq)$, $u\leq s\otimes t$.  Then letting $e=R(s)\wedge D(t)$, $(s|e)\circ (e|t)$ exists and we have that $u\leq (s|e)\circ (e|t)$, so there are $s',t'\in C$ such that $s'\leq s|e\leq s$ and $t'\leq e|t\leq t$, with $u=s'\circ t'=s'\otimes t'$.   Hence the ordered Ehresmann semigroup $(C,\otimes,D,R,\leq)$ satisfies (OS7).
\end{proof}

From the above together with Propositions \ref{deBarros} and \ref{arrows}, we have the following.
 
\begin{cor}  \label{corresp4cor}
There is a category isomorphism between the category of de Barros semigroups and the category of Ehresmann-ordered categories satisfying (OC4).
\end{cor}

In \cite{law1}, Lawson considers those Ehresmann categories in which $\leq_e=\leq_l\circ \leq_r$ satisfies (OC3): he calls them {\em de Barros} categories, and shows that they correspond to de Barros semigroups, as a special case of his main theorem describing the correspondence of Ehresmann semigroups and categories.  Lawson characterises de Barros categories as Ehresmann categories $(C,\circ, D,R,\leq_l,\leq_r)$ such that $(C,\circ, D,R,\leq_e)$ is a totally, regularly ordered category, that is, it is an Ehresmann-ordered category satisfying (OC4), and hence also (OC7).  

From the first part of Theorem \ref{corresp4}, we immediately have the following.

\begin{cor}  
Every Ehresmann-ordered category $(C,\circ, D,R,\leq)$ satisfying (OC4) (that is, every totally, regularly ordered category) is such that $(C,\circ,D,R,\leq_l,\leq_r)$ is de Barros in the sense of \cite{law1}, if we define $s\leq_l t$ providing $s=D(s)|t$ and $s\leq_r t$ providing $s=t|R(s)$, and moreover $\leq$ equals $\leq_e$.
\end{cor}

So de Barros categories in the sense of \cite{law1} are nothing but arbitrary totally, regularly ordered categories $(C,\circ,D,R,\leq)$ in which $\leq_l,\leq_r$ are defined as in the above corollary, and $\leq$ equals $\leq_l\circ \leq_r$.

Corollary \ref{corresp4cor} has further specialisations.  In \cite{law1}, the author specialised his result relating de Barros semigroups and categories to one involving (two-sided) restriction semigroups equipped with their natural order $\leq_e$ and what he called {\em inductive$_1$ categories}: these are $\Omega$-structured categories satisfying (OC8) (hence also (OC4), (OC7) and (OCI) by Proposition 2.8 in \cite{law1}, or equivalently by Lemma \ref{oc}, (OC4A), (OC4B) and (OC6)), and such that $(D(C),\leq)$ is a meet-semilattice; hence they are nothing but Ehresmann-ordered categories satisfying (OC8).  This result can also be obtained from our results, using the fact that (OC4A) and/or (OC4B) hold in an Ehresmann-ordered category if and only if (OS4A) and/or (OS4B) hold in the corresponding ordered Ehresmann semigroup (that is, it is a one or two-sided restriction semigroup).  Together with Proposition \ref{rest}, this gives the following.

\begin{cor}
There is a category isomorphism between the category of restriction semigroups and the category of inductive$_1$ categories.
\end{cor}

Also a consequence is the following result for left restriction semigroups with range; there is an easy dual.

\begin{cor}
There is a category isomorphism between the category of left restriction semigroups with range and the category of Ehresmann-ordered categories satisfying (OC4A).
\end{cor}

We may further specialise this case.

\begin{cor}
There is a category isomorphism between the category of functional left restriction semigroups with range and the category of Ehresmann-ordered categories satisfying (OC4A) in which every element is an epimorphism.
\end{cor}
\begin{proof}
Suppose $(S,\cdot,D,R)$ is a functional left restriction semigroup with range.  Then in the category $(S,\circ,D,R)$, if $s\circ t=s\circ u$ then $st=su$ so $R(s)t=R(s)u$ and so $t=u$ since $R(s)=D(t)=D(u)$.

Conversely, suppose $C$ is an Ehresmann-ordered category satisfying (OC4A) in which every element is an epimorphism.  Then for $s,t,u\in C$, if $s\otimes t=s\otimes u$ then $(s|e)\circ (e|t)=(s|e)\circ (e|u)$ where $e=R(s)\wedge D(t)=R(s)\wedge D(u)$ since $D(t)=D(u)$, so $e|t=e|u$, and so $e\otimes t=e\otimes u$, or $R(s)\otimes t=(R(s)\otimes D(t))\otimes t=(R(s)\otimes D(u))\otimes u=R(s)\otimes u$.  So $(C,\otimes,D,R)$ is a functional left restriction semigroup with range.
\end{proof}

\section{Open Problems}

Two subvarieties of Ehresmann semigroups are those of left restriction semigroups with range and right restriction semigroups with domain, the meet being the variety of restriction semigroups.  Each of these two varieties is contained in the variety of de Barros semigroups, whence so is their join, but it is not known if this join equals the class of de Barros semigroups.

\noindent T. Stokes\\
Department of Mathematics\\ 
The University of Waikato\\
Hamilton, New Zealand\\
phone: +64 7 8384131\\
email: tim.stokes@waikato.ac.nz


\begin{thebibliography}{99}
\bibitem{debar}  de Barros, C. M. (1969). Sur les cat\'{e}gories ordonn\'{e}es r\'{e}guli\`{e}res. {\em  Cahiers Topologie G\'{e}om.  Diff\'{e}rentielle} 11:23--55. 
\bibitem{Ehres} Ehresmann, C. 1980--1984).  Oeuvres compl\`{e}tes et comment\'{e}es (A. C. Ehresmann, Ed.). {\em Suppl. Cahiers Top G\'{e}om. Diff., Amiens.}
\bibitem{fitzkin} FitzGerald, D.G., Kinyon, M.K. (2021). Trace- and pseudo-products: restriction-like semigroups with a band of projections. {\em Semigroup Forum} 103:848--866.
\bibitem{law1} Lawson, M.V. (1991). Semigroups and ordered categories I: the reduced case. {\em J. Algebra} 141:422--462. 
\bibitem{law2} Lawson, M.V. (1991). On Ehresmann semigroups. {\em Semigroup Forum} 103:953--965.\bibitem{scheinDR}  Schein, B.M. (1970). Restrictively multiplicative algebras of transformations. {\em Izv. Vys\v s. U\v cebn. Zaved.}  {1970} no. 4 (95):91--102. [Russian].
\end{thebibliography}
\end{document}